\documentclass[a4paper]{amsart}
\usepackage{amssymb} 

\usepackage[normalem]{ulem}
\usepackage{cancel}

\usepackage[utf8]{inputenc}
\usepackage[T1]{fontenc} 

\usepackage{amssymb,amsmath,amsthm}
\usepackage{geometry}
\usepackage{color}

\usepackage[hidelinks]{hyperref} 
   \hypersetup{final} 

\usepackage{mathtools}
\mathtoolsset{centercolon} 

\newtheorem{lem}{Lemma}
\newtheorem{prop}[lem]{Proposition}
\newtheorem{thm}[lem]{Theorem}
\newtheorem{cor}[lem]{Corollary}
\newtheorem{conj}[lem]{Conjecture}
\theoremstyle{remark}
\newtheorem{rem}[lem]{Remark}

\theoremstyle{definition}

\def\Ex{{\mathbb E}}
\def\Pr{{\mathbb P}}
\def\er{{\mathbb R}}

\def\ve{\varepsilon}

\def\Log{\operatorname{Log}}

\def\dist{d}

\makeatletter
\@namedef{subjclassname@2020}{%
  \textup{2020} Mathematics Subject Classification}
\makeatother

\title[Chevet-type inequalities for subexponential Weibull variables]{Chevet-type inequalities for subexponential Weibull variables and estimates for norms of random matrices}

\author[R. Lata{\l}a]{Rafa{\l} Lata{\l}a}
\address{Institute of Mathematics, University of Warsaw, Banacha 2, 02--097 Warsaw, Poland.}
\email{rlatala@mimuw.edu.pl}

\author[M. Strzelecka]{Marta Strzelecka}
\address{Institute of Mathematics, University of Warsaw, Banacha 2, 02--097 Warsaw, Poland.}
\email{martast@mimuw.edu.pl}


\begin{document}

\begin{abstract}
We prove two-sided Chevet-type inequalities for independent symmetric Weibull random variables with shape parameter $r\in[1,2]$. 
We apply them to provide  two-sided estimates for operator norms  from $\ell_p^n$ to $\ell_q^m$ of random matrices $(a_ib_jX_{i,j})_{i\le m, j\le n}$, in the case when $X_{i,j}$'s are iid symmetric Weibull variables with shape parameter $r\in[1,2]$ or when $X$ is an isotropic log-concave unconditional random matrix. 
We also show how these Chevet-type inequalities imply two-sided bounds for maximal norms from $\ell_p^n$ to $\ell_q^m$ of submatrices of $X$ in both Weibull and log-concave settings.  
\end{abstract}

\maketitle


\section{Introduction and main results}

\subsection{Chevet-type two-sided bounds}

The classical Chevet inequality \cite{Ch} is a two-sided bound for operator norms of Gaussian random 
matrices with iid entries. It states that if $g_{i,j}$, $g_i$, $i,j\geq 1$ are iid standard Gaussian 
random variables,
then  for every pair of nonempty bounded   sets 
$S\subset \er^m$, $T\subset \er^n$ we have
\begin{equation}
\label{eq:Chevet}
\Ex\sup_{s\in S,t\in T}\sum_{i\leq m,j\leq n}g_{i,j}s_it_j\sim
\sup_{s\in S}\|s\|_2\Ex\sup_{t\in T}\sum_{j=1}^ng_{j}t_j
+\sup_{t\in T}\|t\|_2\Ex\sup_{s\in S}\sum_{i=1}^mg_{i}s_i.
\end{equation}
Here and in the sequel we write $f\lesssim g$ or $g\gtrsim f$, if $ f\le Cg$ for a universal constant $C$, 
and $f\sim g$ if $f\lesssim g \lesssim f$. (We write $\lesssim_{ \alpha}$, $\sim_{K,\gamma}$, etc.\ if 
the underlying constant depends on the parameters given in the subscripts.)
The original motivation for Chevet's result was convergence of Gaussian random sums in tensor
spaces, but  Chevet-type bounds (i.e., two-sided bounds allowing to compare the expected double supremum 
of the linear combination of some random variables with some simpler quantities, involving only 
expectations of a single supremum) are also a useful tool in providing bounds for operator norms of random 
matrices, as we shall see in the next subsection. 

Chevet's inequality was generalised to several settings. A version for iid stable r.v.'s was provided in \cite{GMZ}.
Moreover,  it was shown in \cite[Theorem~3.1]{ALLPT} that the upper bound in \eqref{eq:Chevet} holds if 
one replaces  -- on both sides of \eqref{eq:Chevet} -- iid Gaussians by iid symmetric exponential r.v.'s.  
It is however not hard to see (cf. Remark~3 following Theorem~3.1 in \cite{ALLPT}) that such a
bound cannot be reversed.

The first result of this note is a counterpart of \eqref{eq:Chevet}  for symmetric 
Weibull matrices $(X_{i,j})$ with a fixed (shape)  parameter $r\in [1,2]$, i.e., 
symmetric random variables $X_{i,j}$ such that 
\[
\Pr(|X_{i,j}|\geq t)=\exp(-t^r) \quad \mbox{for every } t\geq 0.
\]
It is natural to consider Weibull r.v.'s since they  interpolate between Gaussian and 
exponential r.v.'s -- the case $r=1$ corresponds to exponential r.v.'s, whereas in the case $r=2$
the r.v.'s $X_{i,j}$ are comparable to Gaussian r.v.'s with variance $1/2$ 
(see Lemma~\ref{lem:WeibGauss} below). In particular, our result in the case  $r=1$ provides two-sided bounds 
for iid exponential r.v.'s, which is therefore a better version of the aforementioned upper bound obtained in \cite{ALLPT} (in this note $\rho^*$  denotes the H{\"o}lder conjugate of $\rho\in [1,\infty]$, 
i.e., the unique element of $[1,\infty]$ satisfying $\frac 1\rho+\frac 1{\rho^*}=1$.).

\begin{thm}
\label{thm:chevetWeibull}
Let  $X_{i,j}$, $X_i$, $X_j$, $1\leq i\leq m$, $1\leq j\leq n$
be iid symmetric Weibull r.v.'s with  parameter $r\in [1,2]$.
Then for every  nonempty  bounded sets $S\subset\er^m$ and $T\subset \er^n$ we have
\begin{align*}
\Ex\sup_{s\in S,t\in T}\sum_{i\le m, j\le n} X_{i,j}s_it_j
\sim\ &\sup_{s\in S}\|s\|_{r^*}\Ex\sup_{t\in T}\sum_{j=1}^nX_{j}t_j
+\sup_{t\in T}\|t\|_{r^*}\Ex\sup_{s\in S}\sum_{i=1}^mX_{i}s_i
\\
&\qquad+
\Ex\sup_{s\in S,t\in T}\sum_{i\le m, j\le n} g_{i,j}s_it_j
\\
\sim\ &
\sup_{s\in S}\|s\|_{r^*}\Ex\sup_{t\in T}\sum_{j=1}^nX_{j}t_j
+\sup_{s\in S}\|s\|_2\Ex\sup_{t\in T}\sum_{j=1}^ng_{j}t_j
\\
&\qquad+\sup_{t\in T}\|t\|_{r^*}\Ex\sup_{s\in S}\sum_{i=1}^mX_{i}s_i
+\sup_{t\in T}\|t\|_2\Ex\sup_{s\in S}\sum_{i=1}^m g_{i}s_i.
\end{align*}
\end{thm}

This result generalizes to the case of independent $\psi_r$ random variables. There are several equivalent definitions of $\psi_r$ r.v.'s -- in this paper we say that a random variable $Z$ 
is $\psi_r$  with constant $\sigma$ if
\[
\Pr(|Z|\geq t) \leq  2e^{-(t/\sigma)^r} \qquad \text{for every } t\geq 0.
\]
One of the reasons to investigate Weibull r.v.'s is that 
Weibulls with parameter $r$ are extremal in the class of $\psi_r$ random variables, which appear frequently in probability theory, statistics, and their applications, e.g., in convex geometry (see, e.g.,  \cite{BGVVbook,CGLP2012,V2018}). In particular, Theorem \ref{thm:chevetWeibull} and a standard  estimate (see Lemma~\ref{lem:Weibull-to-psi_r}  below)  yield the following result (observe that we do not assume that the r.v.'s $Y_{i,j}$ are  identically distributed).

\begin{cor}
\label{cor:psi-r-Chevet}
Let  $X_1, X_2, \ldots$  be  iid symmetric Weibull r.v.'s with  parameter $r\in [1,2]$, 
and let $Y_{i,j}$, $1\leq i\leq m$, $1\leq j\leq n$ be independent centered $\psi_r$ 
random variables with constant $\sigma$. 
Then for every   bounded nonempty sets $S\subset\er^m$ and $T\subset \er^n$ we have
\begin{align*}
\Ex\sup_{s\in S,t\in T}\sum_{i\le m, j\le n} Y_{i,j}s_it_j
\lesssim\ &  \sigma \biggl(
\sup_{s\in S}\|s\|_{r^*}\Ex\sup_{t\in T}\sum_{j=1}^nX_{j}t_j
+\sup_{s\in S}\|s\|_2\Ex\sup_{t\in T}\sum_{j=1}^ng_{j}t_j
\\
&\qquad+\sup_{t\in T}\|t\|_{r^*}\Ex\sup_{s\in S}\sum_{i=1}^mX_{i}s_i
+\sup_{t\in T}\|t\|_2\Ex\sup_{s\in S}\sum_{i=1}^m g_{i}s_i  \biggr).
\end{align*}
\end{cor}

Let us now focus on the case $r=1$.
Random vectors with independent symmetric exponential coordinates are extremal in the class
of unconditional isotropic log-concave random vectors  (cf., \cite{BN,LaWTD}). Recall that we call a 
random vector $Z$ in $\er^k$ log-concave, if for any compact nonempty sets $K,L\subset \er^k$ and 
$\lambda\in [0,1]$ we have
\[
\Pr\bigl(Z\in \lambda K+(1-\lambda)L\bigr)\geq \Pr(Z\in K)^{\lambda}\Pr(Z\in L)^{1-\lambda}.
\]
Log-concave vectors are a natural generalization of the class of uniform distributions over 
convex bodies and they are widely investigated in convex geometry and high dimensional 
probability (see the monographs \cite{AAGMbook,BGVVbook}). By the result of Borell \cite{Borell}  we know that 
log-concave vectors with nondegenerate covariance matrix are exactly the vectors with a log-concave 
density, i.e., with a density whose logarithm is a concave function with values in $[-\infty,\infty)$.  

A random vector $Z$ in $\er^k$ is called unconditional, if for every choice of signs $\eta\in\{-1,1 \}^k$ the vectors $Z$ and $(\eta_iZ_i)_{i\le k}$ are equally distributed (or, equivalently, that $Z$ and $(\ve_iZ_i)_{i\le k}$ are equally distributed, where $\ve_1, \ldots ,\ve_k$ are iid symmetric Bernoulli variables independent of $Z$). A~random vector is called isotropic if it is centered and its covariance matrix is the identity.

\cite[Theorem~2]{LaWTD} yields   that for every
bounded nonempty set $U$ in $\er^k$ (see Lemma \ref{lem:Weibull-to-unclogconc} below for a standard reduction to the case of symmetric index sets) and every $k$-dimensional unconditional  isotropic log-concave  
random vector $Y$, 
\begin{equation}	
\label{eq:comp-log-concave-exponential}
\Ex \sup_{u\in U}\sum_{i=1}^k u_iY_i \lesssim\, \Ex \sup_{u\in U}\sum_{i=1}^k u_iE_i,
\end{equation}
where $E_1,E_2,\ldots,E_k$ are independent symmetric exponential r.v.'s (i.e., iid Weibull r.v.'s with shape parameter $r=1$). Hence, Theorem \ref{thm:chevetWeibull} yields the
following Chevet-type bound for isotropic unconditional log-concave random matrices.

\begin{cor}
\label{cor:unclogconc-Chevet}
Let  $E_1, E_2, \ldots$  be  iid symmetric exponential r.v.'s, 
and let $Y=(Y_{i,j})_{1\leq i\leq m,1\leq j\leq n}$ be a random matrix with isotropic unconditional log-concave distribution on $\er^{mn}$.
Then for every  bounded nonempty sets $S\subset\er^m$ and $T\subset \er^n$ we have
\begin{align*}
\Ex\sup_{s\in S,t\in T}\sum_{i\le m, j\le n} Y_{i,j}s_it_j
\lesssim\ &  
\sup_{s\in S}\|s\|_{\infty}\Ex\sup_{t\in T}\sum_{j=1}^nE_{j}t_j
+\sup_{s\in S}\|s\|_2\Ex\sup_{t\in T}\sum_{j=1}^ng_{j}t_j
\\
&\qquad+\sup_{t\in T}\|t\|_{\infty}\Ex\sup_{s\in S}\sum_{i=1}^mE_{i}s_i
+\sup_{t\in T}\|t\|_2\Ex\sup_{s\in S}\sum_{i=1}^m g_{i}s_i.
\end{align*}
\end{cor}

Let us stress that estimate \eqref{eq:comp-log-concave-exponential} is no longer true for general isotropic log-concave  vectors as  \cite[Theorem~5.1]{ALLPT} shows. We do not know  whether there exists a counterpart of Corollary~\ref{cor:unclogconc-Chevet} 
for arbitrary isotropic log-concave random matrices.

The next subsection reveals how our Chevet-type inequalities imply precise bounds for norms of random matrices.

\subsection{Norms of random matrices}
Initially motivated by mathematical physics, the theory of random matrices \cite{AGZbook,TaoMatrBook} is now used in many areas of mathematics. 
A great effort was made to understand the asymptotic behaviour of the edge of the spectrum of  random matrices with 
independent entries. In particular, numerous bounds on their spectral norm (i.e., the largest singular value)
were derived.  The seminal result of Seginer  \cite{Seginer} states that in the iid case 
the expectation of the 
spectral norm is of the same order as the expectation of the maximum Euclidean norm of rows and columns of a given random matrix. 
We know from \cite{LvHY} that the same is true for the structured Gaussian matrices $G_A=(a_{i,j}g_{i,j})_{i\le m, j\le n}$, 
where $g_{i,j}$'s are iid standard Gaussian r.v.'s, and $(a_{i,j})_{i,j}$ is a deterministic matrix encoding the covariance 
structure of $G_A$. 
Although in the structured Gaussian case we still assume that the entries are independent, obtaining optimal bounds 
in this case was much more challenging than in the non-structured case. Upper bounds for the spectral norm of some Gaussian random matrices with dependent entries were obtained very recently in \cite{BBvH2023}. 
	
In this note we are interested in bounding more general operator norms of random matrices.
 For $\rho\in [1,\infty)$ by $\|x\|_{ \rho} = (\sum_{i}|x_i|^{ \rho})^{1/{ \rho}}$,
we denote the $\ell_\rho$-norm of a vector $x$. 
A~similar notation, $\|S\|_\rho = (\Ex|S|^\rho)^{1/\rho}$ is used for  the $L_\rho$-norm of a random variable $S$.  For $\rho=\infty$ we write $\|x\|_{\infty}:=\max_{i}|x_i|$.
By $B_{\rho}^k$ we denote the unit ball in $(\er^k, \|\cdot\|_\rho)$.
 For an $m\times n$ matrix $X=(X_{i,j})_{i\le m, j\le n}$ we denote by
\[
\|X\|_{\ell_p^n\to \ell_q^m} 
= \sup_{t\in B_p^n} \|Xt\|_q 
= \sup_{t\in B_p^n, s\in B_{q^*}^m} s^TXt 
= \sup_{t\in B_p^n, s\in B_{q^*}^m} \sum_{i\le m, j\le n} X_{i,j}s_it_j
\]
 its operator norm from $\ell_p^n$ to $\ell_q^m$. 
 In particular, $\|X\|_{\ell_2^n\to \ell_2^m} $ is the spectral norm of $X$.
When $(p,q)\neq(2,2)$, the moment method used to upper bound the operator norm cannot be 
employed.  This is one of the reasons why upper bounds for $\Ex \|X\|_{\ell_p^n\to \ell_q^m}$ 
are known only in some special cases,  and most of them  are  optimal only up to logarithmic
factors. 
Before we move to a brief survey of these results, let us note that bounds for 
$\Ex \|X\|_{\ell_p^n\to \ell_q^m} $  yield both tail bounds for $\|X\|_{\ell_p^n\to \ell_q^m} $ 
and bounds for $(\Ex \|X\|_{\ell_p^n\to \ell_q^m}^{\rho})^{1/\rho}$ for every $\rho\ge 1$, 
provided that the entries of $X$ satisfy a mild regularity assumption; see 
\cite[Proposition~1.16]{APSS} for more details.

Chevet's inequality together with, say, Remark~\ref{rem:estellpGauss} below easily yields the 
following two-sided estimate for $\ell_p^n\to \ell_q^m$ norms  of iid Gaussian matrices 
for every $p,q\in[1,\infty] $,
\begin{align}
\label{eq:lplqgauss}
\Ex\bigl\|(g_{i,j})_{i\leq m,j\leq n}\bigr\|_{\ell_p^n\rightarrow\ell_q^m}
&\sim 
\begin{cases}
m^{1/q-1/2}n^{1/p^*}+n^{1/p^*-1/2}m^{1/q},&p^*,q\leq 2,
\\
 \sqrt{p^*\wedge \Log n}\:n^{1/p^*}m^{1/q-1/2}+m^{1/q},&q\leq 2\leq p^*,
\\
n^{1/p^*}+ \sqrt{q\wedge \Log m}\: m^{1/q}n^{1/p^*-1/2},&p^*\leq 2\leq q,
\\
\sqrt{p^*\wedge \Log n}\: n^{1/p^*}+\sqrt{q\wedge \Log m}\: m^{1/q},&2\leq q,p^*
\end{cases}
\\ 
\notag
&\sim \sqrt{p^* \wedge \Log n}\: m^{(1/q-1/2)\vee 0} n^{1/p^*} +  
\sqrt{q \wedge \Log m}\: n^{(1/p^*-1/2)\vee 0} m^{1/q},
\end{align}   
where to simplify the notation we define 
\[
\Log n=\max\{1,\ln n\}.
\]

If the entries $X_{i,j}$ are bounded and centered, then it is  known that
\begin{equation}
\label{eq:Radsimple}
\Ex\bigl\|(X_{i,j})_{i\leq m,j\leq n}\bigr\|_{\ell_p^n\rightarrow\ell_q^m}\lesssim_{p,q}
\left\{
\begin{array}{ll}
m^{1/q-1/2}n^{1/p^*}+n^{1/p^*-1/2}m^{1/q},&p^*,q\leq 2,
\\
m^{1/q-1/2}n^{1/p^*}+m^{1/q},&q\leq 2\leq p^*,
\\
n^{1/p^*}+n^{1/p^*-1/2}m^{1/q},&p^*\leq 2\leq q,
\\
 n^{1/p^*}+m^{1/q},&2\leq p^*,q. 			
\end{array}
\right. 
\end{equation}
This was proven in \cite{BGN} in the case $p=2\le q$ and may be easily extrapolated to the whole 
range $1\le p,q\le \infty$ (see \cite{Bennett,CarlMaureyPuhl}, cf., \cite[Remark 4.2]{APSS}).
Moreover, in the case of matrices with iid symmetric Bernoulli r.v.'s inequality \eqref{eq:Radsimple} may be reversed. 
In \cite[Lemma~172]{Naor}   it was shown that in the square case (i.e., when $m=n$)  estimate \eqref{eq:Radsimple}
holds with a constant non depending on $p$   and~$q$. The two-sided estimate  for 
rectangular Bernoulli matrices is more complicated -- we capture the correct dependence 
of the underlying constants on $p$ and $q$ in an upcoming article \cite{LSiidmatr}. As for the Gaussian random matrices, 
also the Bernoulli \textit{structured} case is much more difficult to deal with, even when $p=q=2$. 
Nevertheless, in this case an upper bound optimal up to $\log\log$ factor is known in the case of circulant matrices
due to \cite[Theorem~1.3]{latala-swiatkowski2021norms}.

The case of structured Gaussian matrices in the range $p\le 2\le q$ was investigated in \cite{GHLP}; in this case
\begin{align*}
\nonumber
\Ex \|G_A\|_{ \ell_{p}^n\to\ell_q^m}
\sim_{p,q}
\max_{j \leq n}\|(a_{i,j})_{i=1}^m\|_q 
+ (\Log m)^{1/q}\Bigl(
\max_{i \leq m}\|(a_{i,j})_{j=1}^n\|_{p^*} 
+  \Ex\max_{i\leq m, j\leq n}|a_{i,j}g_{i,j}|\Bigr).
\end{align*}
Since in the range $p\le 2\le q$ we have
\[
\|(a_{i,j})_{i=1}^m\|_q + \|(a_{i,j})_{j=1}^n\|_{p^*} + \Ex\max_{i\leq m, j\leq n}|a_{i,j}g_{i,j}|
\sim_{p,q} 
\Ex\max_{i\le m} \| (a_{i,j}g_{i,j})_j\|_{p^\ast} +\Ex\max_{j\le n}  \| (a_{i,j}g_{i,j})_i\|_q 
\]
(see \cite[Remark~1.1]{APSS}), it seems natural to expect, that, as in the case $p=q=2$, 
\[
\Ex \|G_A\|_{ \ell_{p}^n\to\ell_q^m} \operatorname{\sim}^{\hspace{-5pt}?}_{p,q}
\Ex\max_{i\le m} \| (a_{i,j}g_{i,j})_j\|_{p^\ast} +\Ex\max_{j\le n}  \| (a_{i,j}g_{i,j})_i\|_q .
\]
However, this bound fails outside the range $p\le 2\le q$ (see \cite[Remark~1.1]{APSS}) and, 
as discussed in \cite{APSS}, 
a more reasonable guess is that
\begin{equation} 
\label{eq:conjAPSS}
\Ex \|G_A\colon \ell_{p}^n\to\ell_q^m\|
\operatorname{\sim}^{\hspace{-5pt}?}_{p,q} D_1+D_2 + D_3,
\end{equation}
where 
\begin{equation*}
\begin{split}
D_1  &\coloneqq \|(a_{i,j}^2)_{i\le m, j\le n}\colon \ell^n_{p/2} \to \ell^m_{q/2}\|^{1/2},\\
D_2  &\coloneqq \|(a_{j,i}^2)_{j\le n, i\le m} \colon \ell^m_{q^*/2} \to \ell^n_{p^*/2}\|^{1/2},
\end{split}
\qquad\qquad
\begin{split}
b_j&\coloneqq \|(a_{i,j})_{i\le m} \|_{2q/(2-q)}, \\
d_i &\coloneqq \|(a_{i,j})_{j\le n} \|_{2p/(p-2)},
\end{split}
\end{equation*}
\begin{equation*}
D_3=
\begin{cases}
  \Ex \max_{i \le m,j\le n} |a_{i,j}g_{i,j}| &  \text{if }\  p\leq 2\leq q,\\
  \max_{j\le n}\sqrt{\ln (j+1)} b_j^{*}  &    \text{if }\ p\leq q\leq 2,\\
  \max_{i\le m}\sqrt{\ln (i+1)} d_i^{*}&   \text{if } \ 2\leq p \leq q,\\
  0 &   \textrm{if } \ q<p,
\end{cases} 
\end{equation*}
and $(c_i^*)_{i=1}^k$ is the nonincreasing rearrangement of $(|c_i|)_{i=1}^k$.
It is known by \cite[(1.13) and Corollary~1.4]{APSS}  that \eqref{eq:conjAPSS} holds up to logarithmic terms.
	
It seems that proving the correct asymptotic bound for the operator norm from $\ell_p$ to $\ell_q$ of a~structured Gaussian 
is a challenge. All the more, there is currently no hope  of getting  two-sided bounds in a general case of the structured 
matrices $(a_{i,j}X_{i,j})_{i\le m,j\le n}$ for a wider class of iid random variables $X_{i,j}$.
Therefore, in this paper we restrict ourselves to a special class of variance structures  $(a_{i,j})_{i\le m,j\le n}$: 
the tensor structure. In other words, we  assume that the  structure has a tensor form $a_{i,j}=a_ib_j$ for some 
$a\in \er^m$ and $b\in \er^n$. 
In this case Chevet-type bounds stated in Theorem ~\ref{thm:chevetWeibull}  allow us to provide  two-sided bounds --- with constants 
independent of $p$ and $q$ --- for exponential, Gaussian, and, more general, Weibull tensor structured random matrices.
Since these bounds  have  quite complicated forms we postpone the exact formulations to Section~\ref{sect:tensorweights}. Let us only 
announce here two corollaries from these bounds. 
The first one is an affirmative answer to conjecture~\eqref{eq:conjAPSS}  in the case when 
$(a_{i,j})$ has a tensor form (see Corollary \ref{cor:conj-tensor-case} below). 
In Section~\ref{sect:tensorweights} we  state also a counterpart of 
this conjecture for weighted Weibull matrices and verify it in the tensor case.
Moreover, using \eqref{eq:comp-log-concave-exponential} and a bound for $\Ex \|(a_ib_jE_{i,j})\|_{\ell_p^n \to \ell_q^m}$
we provide a two-sided bound for weighted unconditional isotropic log-concave random matrices $(a_{i,j}Y_{i,j})$ in the tensor case 
$a_{i,j}=a_ib_j$ (see Corollary \ref{cor:tensor-log-concave}  below). We do not know whether a  similar bound holds without the
 unconditionality assumption.
 
Let us now move to another application of Theorem~\ref{thm:chevetWeibull}.
The authors of \cite{ALLPT} used their Chevet-type bound to provide upper bounds for maximal spectral norms of
$k\times l$ submatrices of unconditional  isotropic log-concave random matrices (which turn out to be sharp in the
case of independent exponential entries). Our improved Chevet-type inequality allows us to extend this result and derive
two-sided bounds for maximal $\ell_p \to \ell_q$ norms of submatrices. Let us first formulate the result for Weibull
matrices. By $l_p^J$ we denote the space $\{(x_j)_{j\in J}\colon \sum_{j\in J} |x_j|^p\le 1 \}$ 
equipped with the norm $\|x\|_p\coloneqq ( \sum_{j\in J} |x_j|^p)^{1/p}$.

\begin{thm}
\label{thm:submatricespsir}
Let $r\in [1,2]$ and $(X_{i,j})_{i\le m, j\le n}$ be independent, centered, $\psi_r$ random variables with constant $\sigma$.
Then for any $1\leq k\leq m$, $1\leq l\leq n$ and $p,q\in [1,\infty]$,
\begin{align*}
\Ex \sup_{I, J} \bigl\| (X_{i,j})_{i\in I, j\in J} \bigr\|_{\ell_p^J\to \ell_q^I}
\lesssim \sigma\Bigl(
&k^{(1/q-1/r)\vee 0} l^{1/p^*}\Bigl(\Log\Bigl(\frac{n}l\Bigr) \vee (p^*\wedge \Log l) \Bigr)^{1/r}
\\
& + k^{(1/q-1/2)\vee 0}l^{1/p^*}\Bigl(\Log\Bigl(\frac{n}l\Bigr) \vee (p^*\wedge \Log l\Bigr)^{1/2}
\\
&+ l^{(1/p^*-1/r)\vee 0}k^{1/q}\Bigl(\Log\Bigl(\frac{m}k\Bigr) \vee (q\wedge \Log k) \Bigr)^{1/r} 
\\
&+ l^{(1/p^*-1/2)\vee 0}k^{1/q}\Bigl(\Log\Bigl(\frac{m}k\Bigr) \vee (q\wedge \Log k) \Bigr)^{1/2}
\Bigr),
\end{align*}
where the supremum runs over all sets $I\subset[m]$, $J\subset[n]$ such that $|I|=k$ and $|J|=l$.
Moreover, the above bound may be reversed if $(X_{i,j})_{i\le m, j\le n}$ are iid symmetric Weibull r.v.'s
with parameter~$r$.
\end{thm}

Theorem \ref{thm:submatricespsir} applied with $r=1$, and \eqref{eq:comp-log-concave-exponential} yield the following corollary.

\begin{cor}
Let $(Y_{i,j})_{i\le m, j\le n}$ be isotropic log-concave unconditional matrix. Then
 for any $1\leq k\leq m$, $1\leq l\leq n$ and $p,q\in [1,\infty]$,
 \begin{align*}
\Ex \sup_{I,J} \bigl\| (Y_{i,j})_{i\in I, j\in J} \bigr\|_{\ell_p^J\to \ell_q^I}
\lesssim  \, 
& l^{1/p^*}\Bigl(\Log\Bigl(\frac{n}l\Bigr) \vee (p^*\wedge \Log l) \Bigr)
\\
& + k^{(1/q-1/2)\vee 0} l^{1/p^*}\Bigl(\Log\Bigl(\frac{n}l\Bigr) \vee (p^*\wedge\Log l\Bigr)^{1/2}
\\
&+ k^{1/q}\Bigl(\Log\Bigl(\frac{m}k\Bigr) \vee (q\wedge \Log k) \Bigr) 
\\
&+ l^{(1/p^*-1/2)\vee 0}k^{1/q}\Bigl(\Log\Bigl(\frac{m}k\Bigr) \vee (q\wedge \Log k) \Bigr)^{1/2},
\end{align*}
where the supremum runs over all sets $I\subset[m]$, $J\subset[n]$ such that $|I|=k$ and $|J|=l$.
Moreover, the above bound may be reversed if $(Y_{i,j})_{i\le m, j\le n}$ are  iid symmetric exponential r.v.'s.
\end{cor}

Applying Theorem \ref{thm:submatricespsir} with $k=m$ and $l=n$ we derive the following bound which extends \eqref{eq:lplqgauss}
to the case of Weibull matrices (this also follows from Theorem \ref{thm:aibjWeibull} from Section \ref{sect:tensorweights} 
applied with $a_i=b_j=1$). 

\begin{cor}
\label{cor:rectWeibull1}
Let $(X_{i,j})_{i\leq m,i\leq n}$ be iid symmetric Weibull r.v.'s with  parameter $r\in [1,2]$. Then for every $1\leq p,q\leq \infty$,
\begin{align*}
&\Ex\bigl\|(X_{i,j})_{i\leq m,j\leq n}\bigr\|_{\ell_p^n\rightarrow\ell_q^m}
\\
&\sim
\begin{cases}
m^{1/q-1/2}n^{1/p^*}+n^{1/p^*-1/2}m^{1/q},&p^*,q\leq 2,
\\
 (p^*\wedge \Log n)^{1/r}n^{1/p^*}m^{ (1/q-1/r)\vee 0}
+ \sqrt{p^*\wedge \Log n}\: n^{1/p^*}m^{1/q-1/2}+m^{1/q},&q\leq 2\leq p^*,
\\
 n^{1/p^*}+ (q\wedge \Log m)^{1/r}m^{1/q}n^{(1/p^*-1/r)\vee 0}
+ \sqrt{q\wedge \Log m}\: m^{1/q}n^{1/p^*-1/2},&p^*\leq 2\leq q,
\\
 (p^*\wedge \Log n)^{1/r}n^{1/p^*}+ (q\wedge\Log m)^{1/r}m^{1/q},&2\leq p^*,q
\end{cases}
\\ &\sim
 (p^*\wedge \Log n)^{1/r} m^{(1/q-1/r)\vee 0}n^{1/p^*} +\sqrt{p^*\wedge \Log n}\: m^{(1/q-1/2)\vee 0} n^{1/p^*}
 \\ &\hspace{1,5cm}  
+ (q\wedge \Log m)^{1/r} n^{(1/p^*-1/r)\vee 0}m^{1/q} +\sqrt{q\wedge \Log m}\: n^{(1/p^*-1/2)\vee 0} m^{1/q}.
\end{align*}
In particular, if $n=m$ then
\[
\Ex\bigl\|(X_{i,j})_{i,j=1}^n\bigr\|_{\ell_p^n\rightarrow\ell_q^n}\sim
\begin{cases}
n^{1/q+1/p^*-1/2},&p^*,q\leq 2,
\\
 (p^*\wedge q\wedge \Log n)^{1/r} n^{1/(p^*\wedge q)},&  p^*\vee q\geq 2.
\end{cases}
\]
\end{cor}

Lemma~\ref{lem:sup-bp-lq-norm-Weibull} below and the bound 
$\|X_{i,j}\|_\rho = \bigl(\Gamma(\rho/r+1)\bigr)^{1/\rho} \sim (\rho/r)^{1/r}\sim \rho^{1/r}$ imply that the estimates in 
Corollary~\ref{cor:rectWeibull1} are equivalent to
\begin{align}
\label{eq:rectiid}
\Ex\bigl\|(X_{i,j})_{i\leq m, j\leq n}\bigr\|_{\ell_p^n\rightarrow\ell_q^m}\sim
m^{1/q}\sup_{t\in B_p^n}\Bigl\|\sum_{j=1}^nt_jX_{1,j}\Bigr\|_{ q\wedge \Log m}
+n^{1/p^*}\sup_{s\in B_{q^*}^m}\Bigl\|\sum_{i=1}^{m} s_iX_{i,1}\Bigr\|_{ p^*\wedge \Log n}
\end{align}
and, in the square case, to
\begin{equation}
\label{eq:squareiid}
\Ex\bigl\|(X_{i,j})_{i,j=1}^n\bigr\|_{\ell_p^n\rightarrow\ell_q^n}\sim
\begin{cases}
n^{1/q+1/p^*-1/2}\|X_{1,1}\|_2,&p^*,q\leq 2,
\\
n^{1/(p^*\wedge q)}\|X_{1,1}\|_{p^*\wedge q\wedge \Log n},&  p^*\vee q\geq 2.
\end{cases}
\end{equation}
In the upcoming work \cite{LSiidmatr} we show that \eqref{eq:rectiid} and \eqref{eq:squareiid} hold for  a wider class of  
centered iid random matrices satisfying the following mild regularity assumption: there exists $\alpha \ge 1$ such that 
for every $\rho\ge 1$,
\[
\|X_{i,j}\|_{2\rho}\le \alpha\|X_{i,j}\|_{\rho};
\]
this class contains, e.g.,  all log-concave random  matrices with iid entries and iid Weibull random variables 
with shape parameter $r\in(0,\infty]$.

The rest of this paper is organized as follows.
Section~\ref{sect:proofs} contains the proof of Theorem~\ref{thm:chevetWeibull}, Corollary~\ref{cor:psi-r-Chevet}, 
and inequality~\eqref{eq:comp-log-concave-exponential}.
In Section~\ref{sect:tensorweights} we formulate and prove bounds for norms of random matrices in the tensor structured case.
Finally, Section~\ref{sect:submatrices} contains the proof of Theorem~\ref{thm:submatricespsir}.


\section{ Proofs of Chevet-type bounds}	
\label{sect:proofs}

In this section we show how to derive Chevet-type bounds. Then we move to the proofs of  Corollary~\ref{cor:psi-r-Chevet} and inequality~ \eqref{eq:comp-log-concave-exponential}.

\begin{proof}[Proof of Theorem \ref{thm:chevetWeibull}]
The second estimate follows by Chevet's inequality. The proof of the first upper bound is a 
modification of the proof of Theorem~3.1 from \cite{ALLPT}.  

Let us briefly recall the notation from \cite{ALLPT}. 
 For a metric space  $(U, \dist)$  and $\rho>0$ let 
\[
\gamma_\rho(U, \dist) 
:= \inf_{(U_l)_{l=0}^\infty} \sup_{u\in U} \sum_{l=0}^\infty 2^{l/\rho} \dist(u,U_l), 
 \]
where the infimum is taken over all admissible sequences of sets, i.e., all sequences 
$(U_l)_{l=0}^\infty$ of subsets of $U$, such that $|U_0|=1$, and $|U_l|\le 2^{2^l}$ for $l\ge 1$. 
Let $ \dist_{\rho}$ be the $\ell_{\rho}$-metric in the appropriate dimension. 
Since $r\in [1,2]$, by the result of Talagrand \cite{Tal} (one may also use the more general 
\cite[Theorem~2.4]{LaT} to see more explicitely that two-sided bounds hold with constants 
independent of parameter $r$) for every nonempty $U\subset \er^k$,
\begin{equation}
\label{eq:gamma-est-weibull}
\Ex\sup_{u\in U}\sum_{i\le k} u_ig_i \sim \gamma_2(U,  \dist_2) \quad \text{ and } \quad 
\Ex\sup_{u\in U} \sum_{i=1}^k u_iX_i \sim \gamma_r(U,  \dist_{r^*})+\gamma_2(U, \dist_2).
\end{equation}

For nonempty sets $S\subset \er^m$ and $T\subset \er^n$ let 
\[
S\otimes T=\{ s\otimes t \colon s\in S, t\in T \},
\]
where $s\otimes t\coloneqq (s_it_j)_{i\le m, j\le n}$ belongs to the space of real $m\times n$ matrices, which we identify with $\er^{mn}$. 	
Now we will prove that 
\begin{equation}
\label{eq:gamma1-tensor-sep}
\gamma_{r}(S\otimes T, \dist_{r^*}) \sim \sup_{t\in T}\|t\|_{r^*}\gamma_r(S, \dist_{r^*}) 
+ \sup_{s\in S}\|s\|_{r^*}\gamma_r(T, \dist_{r^*}).
\end{equation}
Let $S_l\subset S$ and $T_l\subset T$, $l=0,1,\ldots$ be admissible sequences of sets.  
Set $T_{-1}:=T_0$, $S_{-1}:=S_0$
and define $U_l\coloneqq S_{l-1}\otimes T_{l-1}$. 
Then $(U_l)_{l\ge 0}$ is an admissible sequence of subsets of $S\otimes T$.
	
Note that  for all $s',s''\in S$, and $t',t''\in T$ we have
\begin{align*}
\dist_{r^*} (s'\otimes t',s''\otimes t'' ) 
& = \|s'\otimes t'-s''\otimes t''\|_{r^*} 
 \leq \|s'\otimes(t'-t'')\|_{r^*} + \|(s'-s'')\otimes t''\|_{r^*}
\\ & = \|s'\|_{r^*} \|t'-t''\|_{r^*} + \|t''\|_{r^*}\|s'-s''\|_{r^*} 
\\ & \leq \sup_{s\in S}\|s\|_{r^*}\dist_{r^*}(t',t'') + \sup_{t\in T}\|t\|_{r^*}\dist_{r^*}(s',s'').
\end{align*}
Therefore
\begin{multline*}
\gamma_{r}(S\otimes T,  \dist_{r^*}) 
 \le \sup_{s\in S, t\in T} \sum_{l=0}^\infty 2^{l/r}\dist_{r^*}(s\otimes t, U_l) 
\\
 \le 
\sup_{s\in S}\|s\|_{r^*}\sup_{t\in T} \sum_{l=0}^\infty 2^{l/r} \dist_{r^*}(t, T_{l-1}) 
+\sup_{t\in T}\|t\|_{r^*}\sup_{s\in S} \sum_{l=0}^\infty 2^{l/r} \dist_{r^*}(s, S_{l-1}).
\end{multline*}
Taking the infimum over all admissible sequences $(S_l)_{l\ge 0}$ and $(T_l)_{l\ge 0}$ we get the upper bound~\eqref{eq:gamma1-tensor-sep}.

To establish the lower bound in \eqref{eq:gamma1-tensor-sep} it is enough to observe that
\begin{align*}
\gamma_{r}(S\otimes T, \dist_{r^*})
&\geq
\max\Bigl\{\sup_{t\in T}\gamma_{r}\bigl(S\otimes \{t\}, \dist_{r^*}\bigr),\ 
\sup_{s\in S}\gamma_{r}\bigl(\{s\}\otimes T, \dist_{r^*}\bigr)\Bigr\}
\\
&=\max\Bigl\{\sup_{t\in T}\|t\|_{r^*}\gamma_{r}(S, \dist_{r^*}),\ 
\sup_{s\in S}\|s\|_{r^*}\gamma_{r}(T, \dist_{r^*})\Bigr\}.
\end{align*}

Bounds   \eqref{eq:gamma-est-weibull} and \eqref{eq:gamma1-tensor-sep} imply
\begin{multline}
\label{eq:estchev1}
\Ex\sup_{s\in S,t\in T}\sum_{i\le m, j\le n} X_{i,j}s_it_j
 \sim
\gamma_2(S\otimes T, \dist_2) + \gamma_{r}(S\otimes T, \dist_{r^*}) 
\\ 
 \sim
\Ex\sup_{s\in S,t\in T}\sum_{i\le m, j\le n} g_{i,j}s_it_j + 
\sup_{t\in T}\|t\|_{r^*}\gamma_{r}(S, \dist_{r^*}) 
+ \sup_{s\in S}\|s\|_{r^*}\gamma_{r}(T, \dist_{r^*}).
\end{multline}
Moreover, Chevet's inequality and \eqref{eq:gamma-est-weibull} yield
\begin{align}
\label{eq:estchev2}
\Ex\sup_{s\in S,t\in T}\sum_{i\le m, j\le n} g_{i,j}s_it_j
&\gtrsim \sup_{t\in T}\|t\|_{2}\gamma_{2}(S, \dist_2) 
+ \sup_{s\in S}\|s\|_{2}\gamma_{2}(T, \dist_2)
\\
\notag
&\geq \sup_{t\in T}\|t\|_{r^*}\gamma_{2}(S, \dist_2) 
+ \sup_{s\in S}\|s\|_{r^*}\gamma_{2}(T, \dist_2).
\end{align}
The first  asserted inequality follows by applying \eqref{eq:estchev1}, \eqref{eq:estchev2} and \eqref{eq:gamma-est-weibull}.
\end{proof}

Corollary~\ref{cor:psi-r-Chevet} immediately follows by a  symmetrization, Theorem~\ref{thm:chevetWeibull} and the following standard lemma.
 
 \begin{lem} 
 \label{lem:Weibull-to-psi_r}
Let $X_{i,j}$'s be iid symmetric Weibull r.v.'s with parameter $r\in [1,2]$, and
$Y_{i,j}$'s be independent symmetric $\psi_r$ random variables  with constant $\sigma$.
Then for every  bounded nonempty sets $S\subset\er^m$ and $T\subset \er^n$ we have
\begin{align*}
\Ex\sup_{s\in S,t\in T}\sum_{i\le m, j\le n} Y_{i,j}s_it_j
\le\ & 2 \sigma\Ex\sup_{s\in S,t\in T}\sum_{i\le m, j\le n} X_{i,j}s_it_j.
\end{align*}
 \end{lem}

\begin{proof}
The $\psi_r$ assumption gives $\Pr(|Y_{i,j}|\geq t)\leq  2\Pr(|\sigma X_{i,j}|\geq t)$.
Let $(\delta_{i,j})_{i\le m,j\le n}$ be iid r.v.'s independent of all the others,
such that $\Pr(\delta_i=1)= 1/2=1-\Pr(\delta_i=0)$.
Then $\Pr(|\delta_{i,j}Y_{i,j}|\geq t)\leq \Pr(| \sigma X_{i,j}|\geq t)$ for every $t\geq 0$, 
so we may find such a representation of $(X_{i,j}, Y_{i,j},\delta_{i,j})_{i\le m, j\le n}$, 
that  $\sigma |X_{i,j}|\geq |\delta_{i,j}Y_{i,j}|$ a.s.
Let $(\ve_{i,j})_{i\leq m, j\leq n}$ be a matrix with iid symmetric $\pm 1$ entries (Rademachers) 
independent of all the others. 
Then  the contraction principle and Jensen's inequality imply
\begin{align*}
\sigma \Ex \sup_{s\in S, t\in T} \sum_{i,j}X_{i,j}s_it_j
& =  
\Ex \sup_{s\in S, t\in T} \sum_{i,j}\ve_{i,j}|\sigma X_{i,j}|s_it_j 
 \geq
\Ex \sup_{s\in S, t\in T} \sum_{i,j}\ve_{i,j}|\delta_{i,j}Y_{i,j}|s_it_j
\\
& = \Ex  \sup_{s\in S, t\in T} \sum_{i,j}\delta_{i,j}Y_{i,j}s_it_j
\geq  \Ex \sup_{s\in S, t\in T} \sum_{i,j}Y_{i,j} \Ex\delta_{i,j}s_it_j
\\ &
= \frac{1}{2} \Ex \sup_{s\in S, t\in T} \sum_{i,j}Y_{i,j} s_it_j. \qedhere
\end{align*}
\end{proof}

\begin{lem}
\label{lem:Weibull-to-unclogconc}
Formula \eqref{eq:comp-log-concave-exponential} holds for every
bounded nonempty set $U\subset \er^k$ and every $k$-dimensional unconditional  isotropic log-concave random vector $Y$.
\end{lem}

\begin{proof}
\cite[Theorem~2]{LaWTD} states that for every norm $\|  \cdot\|$ on $\er^k$,
$\Ex\|Y\|\leq C\Ex\|E\|$, where $E=(E_1,\ldots,E_k)$. In other words, 
\eqref{eq:comp-log-concave-exponential} holds for bounded symmetric sets $U$.

Now, let $U$ be arbitrary. Take any point $v\in U$. Since
$\Ex \sum_{i=1}^{ k} v_iY_i=0$ we have
\begin{align*}
\Ex \sup_{u\in U}\sum_{i=1}^k u_iY_i =
\Ex \sup_{u\in U-v}\sum_{i=1}^{k} u_iY_i
\leq \Ex \sup_{u\in U-v}\Bigl|\sum_{i=1}^{k} u_iY_i\Bigr|
\leq C\Ex \sup_{u\in U-v}\Bigl|\sum_{i=1}^{k} u_iE_i\Bigr|,
\end{align*}
where the last inequality follows by \eqref{eq:comp-log-concave-exponential} applied
to the symmetric set $(U-v)\cup (v-U)$. On the other hand, the distribution of $E$ is symmetric, 
$0\in U-v$, and $\Ex \sum_{i=1}^{ k} v_iE_i=0$, so
\begin{align*}
\Ex \sup_{u\in U-v}\Bigl|\sum_{i=1}^{k} u_iE_i\Bigr|
&\leq \Ex \sup_{u\in U-v}\Bigl(\sum_{i=1}^{k} u_iE_i\Bigr)\vee 0
+ \Ex \sup_{u\in U-v}\Bigl(-\sum_{i=1}^{k} u_iE_i\Bigr)\vee 0
\\
&=2\Ex \sup_{u\in U-v}\Bigl(\sum_{i=1}^{k} u_iE_i\Bigr)\vee 0
= 2\Ex\sup_{u\in U-v}\sum_{i=1}^{k} u_iE_i
= 2\Ex\sup_{u\in U}\sum_{i=1}^{k} u_iE_i.\qedhere
\end{align*}

\end{proof}

\section{Matrices $(a_ib_jX_{ij})$} \label{sect:tensorweights}

In this section we shall consider matrices of the form $(a_ib_jX_{i,j})_{i \leq m, j\leq n}$. 
 Before presenting  our results we need to introduce some notation.
 
By $(c_i^*)_{i=1}^k$ we will denote the nonincreasing rearrangement of $(|c_i|)_{i=1}^k$. 
For $\rho\geq 1$ we set
\[
\varphi_\rho(t)
=\begin{cases}
\exp( 2-2t^{-\rho}),& t>0.
\\
0,& t=0,
\end{cases}
\]
and define
\[
\|(c_i)_{i\leq k}\|_{\varphi_\rho}:=\inf\Bigl\{t>0\colon \sum_{i=1}^k\varphi_\rho(|c_i|/t)\leq 1\Bigr\}.
\]
The function $\varphi_\rho$ is not convex on $\er_+$. However, it is increasing and convex
on $[0,1]$ and $\varphi_\rho(1)=1$. So we may find a convex function $\tilde{\varphi}_\rho$ on $[0,\infty)$
such that $\varphi_{\rho}=\tilde{\varphi}_{\rho}$ on $[0,1]$. Then clearly 
$\|\cdot\|_{\varphi_{\rho}}=\|\cdot\|_{\tilde{\varphi}_\rho}$. Thus $\|\cdot\|_{\varphi_{\rho}}$ is an Orlicz norm.

Let us first present the bound in the Gaussian case.
\begin{thm}
\label{thm:aibjgauss}
For every $1\leq p,q\leq \infty$ and deterministic sequences $(a_i)_{i\leq m}$, $(b_j)_{j\leq m}$,
\begin{align*}
&\Ex\bigl\|(a_ib_jg_{i,j})_{i\leq m,j\leq n}\bigr\|_{\ell_p^n\rightarrow\ell_q^m}
\\
& \qquad\sim
\begin{cases}
\|a\|_{\frac{2q^*}{q^*-2}}\|b\|_{p^*}+\|a\|_q\|b\|_{\frac{2p}{p-2}},&p^*,q< 2,
\\
\|a\|_{\frac{2q^*}{q^*-2}}\bigl(\|(b_j^*)_{j\leq e^{p^*}}\|_{\varphi_{2}}
+\sqrt{p^*}\|(b_j^*)_{j> e^{p^*}}\|_{p^*}\bigr)+\|a\|_q\|b\|_{\infty} ,&q< 2\leq p^*,
\\
\|a\|_{\infty}\|b\|_{p^*}+\bigl(\|(a_i^*)_{i\leq e^{q}}\|_{\varphi_{2}}
+\sqrt{q}\|(a_i^*)_{i> e^q}\|_{q}\bigr)\|b\|_{\frac{2p}{p-2}},&p^*< 2\leq q,
\\
\|a\|_{\infty}\bigl(\|(b_j^*)_{j\leq e^{p^*}}\|_{\varphi_{2}}+\sqrt{p^*}\|(b_j^*)_{j> e^{p^*}}\|_{p^*}\bigr)
&
\\
\quad+\bigl(\|(a_i^*)_{i\leq e^{q}}\|_{\varphi_{2}}+\sqrt{q}\|(a_i^*)_{i> e^q}\|_{q}\bigr)\|b\|_{\infty},
&2\leq p^*,q.
\end{cases}
\end{align*}
\end{thm}

Before we move to the Weibull case, let us 
see how Theorem~\ref{thm:aibjgauss} implies conjecture~\eqref{eq:conjAPSS} for the tensor structured Gaussian matrices.

\begin{cor}
\label{cor:conj-tensor-case}
Assume that there exists $a\in \er^m$ and $b\in\er^n$ such that $a_{ij}=a_ib_j$ for every $i\le m,j\le n$. 
Then conjecture~\eqref{eq:conjAPSS} holds.
\end{cor}

\begin{proof}
 If $p^*=\infty$ or $q=\infty$, then \eqref{eq:conjAPSS} is satisfied for an arbitrary matrix 
 $(a_{i,j})_{i,j}$ by  \cite[Remark 1.4]{GHLP},
\cite[Proposition 1.8 and Corollary 1.11]{APSS} 

In the case $p^*,q<\infty$ we shall show that
\begin{equation}
\label{eq:estprod}
D_1+D_2
\lesssim
\Ex\bigl\|(a_ib_jg_{i,j})_{i\leq m,j\leq n}\bigr\|_{\ell_p^n\rightarrow\ell_q^m}
\lesssim  \sqrt{q}D_1+ \sqrt{p^*}D_2.
\end{equation}

The lower bound 
follows by \cite[Proposition~5.1 and Corollary~5.2]{APSS}.

 To establish the upper bound let us first compute $D_1$ and $D_2$ in the case $a_{i,j}=a_ib_j$.
If $p>2$, then $2(p/2)^*=2p/(p-2)$, so  for every $p\in [1,\infty]$, 
\[
D_1=\sup_{t\in B_{p/2}^n} \biggl(\sum_{i=1}^m |a_i|^q \Bigl| \sum_{j=1}^n b_j^2t_j \Bigr|^{q/2} \biggr)^{1/q}
=\|a\|_q \sup_{t\in B_{p/2}^n}  \Bigl | \sum_{j=1}^n b_j^2t_j \Bigr |^{1/2}
= \|a\|_q \begin{cases}
\|b\|_{2p/(p-2)} & p^*<2, \\
\|b\|_\infty & p^*\ge 2,
\end{cases}
\]
and dually 
\[
D_2=
\|b\|_{p^*} \begin{cases}
\|a\|_{2q^*/(q^*-2)} & q<2, \\
\|a\|_\infty & q\ge 2.
\end{cases}
\]
Moreover, 
\begin{align*}
\|(b_j^*)_{j\leq e^{p^*}}\|_{\varphi_{2}}+\sqrt{p^*}\|(b_j^*)_{j> e^{p^*}} \|_{p^*}
&  \leq
 2\sqrt{p^*} b_1^*+ \sqrt{p^*}\|(b_j^*)_{j> e^{p^*}}  \|_{p^*}
 \leq 3\sqrt{p^*}\|(b_j^*)_{j}\|_{p^*},
\end{align*}
and similarly 
\[
\|(a_i^*)_{i\leq e^{q}}\|_{\varphi_{2}}+\sqrt{q}\|(a_i^*)_{i> e^q}\|_{q} \lesssim \sqrt{q}\|(a_i^*)_i\|_q.
\]
Hence, Theorem~\ref{thm:aibjgauss}  yields the upper bound in \eqref{eq:estprod}.
\end{proof}

In the Weibull case we get the following bound.

\begin{thm}
\label{thm:aibjWeibull}
Let $(X_{i,j})_{i\leq m,j\leq n}$ be iid symmetric Weibull r.v.'s with  parameter $r\in [1,2]$. Then for every $1\leq p,q\leq \infty$ and deterministic sequences $a=(a_i)_{i\leq m}$ and
$b=(b_j)_{j\leq n}$,
\begin{align*}
&\Ex\bigl\|(a_ib_jX_{i,j})_{i\leq m,j\leq n}\bigr\|_{\ell_p^n\rightarrow\ell_q^m}
\\
&\quad\sim
\begin{cases}
\|a\|_{\frac{2q^*}{q^*-2}}\|b\|_{p^*}+\|a\|_q\|b\|_{\frac{2p}{p-2}},&p^*,q< 2,
\\
\|a\|_{\frac{2q^*}{q^*-2}}\bigl(\|(b_j^*)_{j\leq e^{p^*}}\|_{\varphi_{2}}
+\sqrt{p^*}\|(b_j^*)_{j> e^{p^*}}\|_{p^*}\bigr)&
\\
\quad+\|a\|_{\frac{r^*q^*}{q^*-r^*}}\bigl(\|(b_j^*)_{j\leq e^{p^*}}\|_{\varphi_{r}}
+(p^*)^{1/r}\|(b_j^*)_{j> e^{p^*}}\|_{p^*}\bigr)
+\|a\|_q\|b\|_\infty ,&q< r,2\leq p^*,
\\
\|a\|_{\frac{2q^*}{q^*-2}}\bigl(\|(b_j^*)_{j\leq e^{p^*}}\|_{\varphi_{2}}
+\sqrt{p^*}\|(b_j^*)_{j> e^{p^*}}\|_{p^*}\bigr)&
\\
\quad+\|a\|_{\infty}\bigl(\|(b_j^*)_{j\leq e^{p^*}}\|_{\varphi_{r}}
+(p^*)^{1/r}\|(b_j^*)_{j> e^{p^*}}\|_{p^*}\bigr)
+\|a\|_q\|b\|_\infty ,&r\leq q< 2\leq p^*,
\\
\|a\|_{\infty}\|b\|_{p^*}+\bigl(\|(a_i^*)_{i\leq e^{q}}\|_{\varphi_{2}}
+\sqrt{q}\|(a_i^*)_{i> e^q}\|_{q}\bigr)\|b\|_{\frac{2p}{p-2}}&
\\
\quad+\bigl(\|(a_i^*)_{i\leq e^q}\|_{\varphi_{r}}+q^{1/r}\|(a_i^*)_{i> e^q}\|_q\bigr)
\|b\|_{\frac{r^*p}{p-r^*}},&p^*<r, 2\leq q,
\\
\|a\|_{\infty}\|b\|_{p^*}+\bigl(\|(a_i^*)_{i\leq e^{q}}\|_{\varphi_{2}}
+\sqrt{q}\|(a_i^*)_{i> e^q}\|_{q}\bigr)\|b\|_{\frac{2p}{p-2}}&
\\
\quad+\bigl(\|(a_i^*)_{i\leq e^q}\|_{\varphi_{r}}+q^{1/r}\|(a_i^*)_{i> e^q}\|_q\bigr)\|b\|_{\infty},&r\leq p^*< 2\leq q,
\\
\|a\|_{\infty}\bigl(\|(b_j^*)_{j\leq e^{p^*}}\|_{\varphi_{r}}
+(p^*)^{1/r}\|(b_j^*)_{j> e^{p^*}}\|_{p^*}\bigr)&
\\
\quad+
\bigl(\|(a_i^*)_{i\leq e^q}\|_{\varphi_{r}}+q^{1/r}\|(a_i^*)_{i> e^q}\|_q\bigr)\|b\|_{\infty},
&2\leq p^*,q.
\end{cases}
\end{align*}
\end{thm}

\begin{cor}	
\label{cor:tensor-log-concave}
Suppose that $r\in[1,2]$, $a\in \er^m$, $b\in \er^n$, and 
$(X_{i,j})_{i\le m, j\le n}$ is a random matrix with independent $\psi_r$ entries with constant $\sigma$
 such that $\Ex|X_{ij}|\geq \gamma$. 
Then 
\begin{equation}
\label{eq:estproduct2}
\gamma(D_1+D_2)
\lesssim
\Ex\bigl\|(a_ib_jX_{i,j})_{i\le m, j\le n}\bigr\|_{\ell_p^n\rightarrow\ell_q^m}
\lesssim \sigma\bigl( q^{1/r}D_1+ (p^*)^{1/r}D_2\bigr).
\end{equation}
Moreover, if $(X_{i,j})_{i\le m, j\le n}$ is an  isotropic log-concave unconditional 	
random matrix, then two-sided estimate \eqref{eq:estproduct2} holds with $r=\sigma=\gamma=1$.
\end{cor}

\begin{proof}
The lower bound follows by the proof of \cite[Proposition~5.1]{APSS} (which in fact shows that the assertion of \cite[Proposition~5.1]{APSS} holds for unconditional random matrices whose entries satisfy $\Ex |X_{ij}|\ge c$). 
To derive the upper bound we proceed similarly as in the proof of Corollary~\ref{cor:conj-tensor-case} using 
Theorem~\ref{thm:aibjWeibull} 
(instead of Theorem~\ref{thm:aibjgauss}) and then apply Lemma~\ref{lem:Weibull-to-psi_r} --- or inequality~\eqref{eq:comp-log-concave-exponential} in the log-concave case.
\end{proof}

Corollary~\ref{cor:tensor-log-concave} suggests, that in a non-tensor case it makes sense to pose the following counterpart of  conjecture  \eqref{eq:conjAPSS}.

\begin{conj}
\label{conj:weighted-Weibull}
Assume that $r\in[1,2]$, $(a_{i,j})_{i\le m,j\le n}$ is a deterministic $m\times n$ matrix, and 
$(X_{i,j})_{i\le m, j\le n}$  be iid symmetric Weibull r.v.'s with parameter $r\in [1,2]$.
Let
\begin{equation*}
D_{3,r}=
\begin{cases}
\Ex \max_{i \le m,j\le n} |a_{i,j} X_{i,j}| &  \text{if }\  p\leq 2\leq q,\\
\max_{j\le n} b_j^{*}\ln^{1/r} (j+1)   &    \text{if }\ p\leq q\leq 2,\\
\max_{i\le m}d_i^{*}\ln^{1/r} (i+1) &   \text{if } \ 2\leq p \leq q,\\
0 &   \textrm{if } \ q<p.
\end{cases} 
\end{equation*}
Is it true that
\[
\Ex \|(a_{i,j}X_{i,j})_{i,j}\colon \ell_{p}^n\to\ell_q^m\|
 \sim_{p,q} \ D_1+D_2 + D_{3,r}\quad?
\]
\end{conj}

\begin{rem}
Using similar methods as in the proofs of \cite[Propositions~1.8 and 1.10]{APSS}, one may show that 
Conjecture~\ref{conj:weighted-Weibull} holds whenever $p\in \{1,\infty\}$ or $q\in \{1,\infty\}$. 
Moreover, it follows by \cite[Theorem~4.4]{LvHY} and a counterpart of \cite[equation~(1.11)]{APSS} for 
iid Weibull r.v.'s that Conjecture~\ref{conj:weighted-Weibull} holds in the case $p=2=q$.
\end{rem}

Now we provide the following lemma yielding the equivalence between  
\eqref{eq:rectiid} and the assertion of  Corollary~\ref{cor:rectWeibull1}.

\begin{lem}	
\label{lem:sup-bp-lq-norm-Weibull}
Let $X_1,X_2,\ldots,X_k$ be iid symmetric Weibull r.v.'s with  parameter  $r\in [1,2]$
and let $ \rho_1,\rho_2 \in [1,\infty)$. Then 
\begin{align*}
\sup_{t\in B^k_{\rho_1}}\Bigl\|\sum_{i=1}^kt_iX_i\Bigr\|_{ \rho_2}
& \sim 
\begin{cases}
 \rho_2^{1/r},&  1\leq \rho_1\leq 2,
\\
 \rho_2^{1/r}+\sqrt{ \rho_2}k^{1/2-1/{ \rho_1}},& 2\leq \rho_1\leq r^*,
\\
 \rho_2^{1/r}k^{1/r^*-1/ \rho_1}+\sqrt{  \rho_2}k^{1/2-1/{ \rho_1}},&  \rho_1\geq r^*
\end{cases}
\\
& \sim  \rho_2^{1/r} k^{\bigl(\frac{1}{ \rho_1^*} - \frac 1r\bigr)\lor 0} 
+\sqrt{ \rho_2} k^{\bigl(\frac 1{ \rho_1^*}-\frac 12\bigr)\lor 0}.
\end{align*}
\end{lem}

\begin{proof}

The Gluskin--Kwapie{\'n} inequality \cite{GK} easily implies that  
\begin{equation}  
\label{eq:comp-moms-sums-Weibull}
\Bigl\|\sum_{i=1}^kt_iX_i\Bigr\|_{\rho_2}\sim  \rho_2^{1/r}\|t\|_{r^*}+ \rho_2^{1/2}\|t\|_2.
\end{equation}
Hence it is enough to observe that
\begin{equation*}	
\sup_{t\in B^k_{ \rho_1}}\|t\|_{ \rho}
=
\begin{cases}
1, &{ \rho_1\leq \rho},
\\
k^{1/\rho-1/{\rho_1}},& { \rho_1\geq \rho}.
\end{cases}
\qedhere
\end{equation*}
\end{proof}

Before proving Theorems~\ref{thm:aibjgauss} and \ref{thm:aibjWeibull} we need to formulate several technical results.
H\"older's inequality yields the following simple lemma.

\begin{lem}
\label{lem:multiplier}
Let $1\leq \rho_1 ,\rho_2\leq \infty$ and $c=(c_i)\in \er^k$. Then
\[
\sup_{t\in B_{\rho_1}^k}\|(c_it_i)\|_{\rho_2}
=
\begin{cases}
\|c\|_{\infty},& \rho_1\leq \rho_2,
\\
\|c\|_{\rho_1\rho_2/(\rho_1-\rho_2)},& \rho_2 <\rho_1<\infty,
\\
\|c\|_{\rho_2},& \rho_1=\infty.
\end{cases}
\]
\end{lem}

The next result is a two-sided bound for the $\ell_\rho$-norms of a weighted sequence of independent
Weibull r.v.'s. 
Much more general two-sided estimates for the Orlicz norms of weighted vectors with iid coordinates were obtained in \cite{GLSW}. 
However, the  formula stated therein is quite involved and not easy to decrypt in the case of $\ell_\rho$-norms. 
Therefore we give an alternative proof in our special setting,  providing a form of the two-sided estimate 
which is more handy for the purpose of proving Theorems~\ref{thm:aibjgauss} and \ref{thm:aibjWeibull}.

\begin{prop}
\label{prop:ellrhoweibull}
Let $(X_{i})_{i\leq k}$ be iid symmetric Weibull r.v.'s with  parameter  $r\in [1,2]$.
Then for every $1\leq \rho\leq  \infty$ and every sequence $c=(c_i)_{i=1}^k$ we have
\begin{equation}
\label{eq:ellrhoweibull}
\Ex\bigl\|(c_iX_i)_{i=1}^k\bigr\|_\rho\sim
\|(c_i^*)_{i\leq e^\rho}\|_{\varphi_{r}}
+\rho^{1/r}\|(c_i^*)_{i> e^\rho}\|_\rho,
\end{equation}
where   $\varphi_r(x)=\exp(2-2x^{-r})$ and $(c_i^*)_{i\leq k}$ is the nonincreasing rearrangement of $(|c_i|)_{i\leq k}$.
\end{prop}

\begin{rem}	
\label{rem:ellsmallrhoweibull}
For $\rho\leq 2$ we have $\Ex\|(c_iX_i)\|_\rho\sim \|c\|_\rho$. It is not hard to deduce this 
from \eqref{eq:ellrhoweibull}. Alternatively one may use the Khintchine--Kahane-type inequality 
$\Ex\|(c_iX_i)\|_\rho\sim (\Ex\|(c_iX_i)\|_\rho^\rho)^{1/\rho}$.
\end{rem}

\begin{proof}[Proof of Proposition \ref{prop:ellrhoweibull}]
First we show that for every $1\leq l \leq k$,
\begin{equation}
\label{eq:ellinftyweibull}
\Ex\bigl\|(c_iX_i)_{i\leq l}\bigr\|_\infty \sim \|(c_i)_{i\leq l}\|_{\varphi_{r}}.
\end{equation}
Let $t= 2^{1/r}\|(c_i)_{ i\leq l}\|_{\varphi_{r}}$.
Then
\[
\sum_{i=1}^{ l}\Pr(|c_iX_i|\geq t)=\sum_{i=1}^{ l}e^{-2}\varphi_r\bigl(2^{1/r}|c_i|/t\bigr)
=e^{-2}.
\]
This and the independence of $X_{i}$'s imply
\[
\Ex\bigl\|(c_iX_i)_{i=1}^{ l}\bigr\|_\infty\geq t\Pr(\max_{ i\leq l}|c_iX_i|\geq t)\gtrsim t.
\]
Moreover, for $u\geq 1$,
\begin{align*}
\Pr(\max_{ i\leq l}|c_iX_i|\geq ut)
&\leq \sum_{i=1}^{ l}\Pr(|c_iX_i|\geq ut)
=\sum_{i=1}^{ l}e^{-2}\varphi_r\bigl(2^{1/r}|c_i|/(tu)\bigr)
\\ & \leq \sum_{i=1}^{ l}e^{-2u^r}\varphi_r\bigl(2^{1/r}|c_i|/t\bigr)
 =e^{-2u^r},
\end{align*}
where the  second inequality follows since $(xu)^r\geq x^r+u^r-1$ for $x,u\geq 1$. Thus
integration by parts yields $\Ex\|(c_iX_i)_{i\leq l}\|_\infty\lesssim t$ and \eqref{eq:ellinftyweibull} follows.

To establish \eqref{eq:ellrhoweibull} for $\rho\in [1,\infty)$ we may and will assume
that $c_1\geq c_2\geq \ldots\geq c_k\geq 0$. Then $c_k^*=c_k$. 

We have by \eqref{eq:ellinftyweibull} applied with $l=\lfloor e^\rho\rfloor \wedge k$,
\[
\Ex\|(c_iX_i)_{i\leq e^\rho}\|_\rho\sim \Ex\|(c_iX_i)_{i\leq e^\rho}\|_\infty
\sim \|(c_i)_{i\leq e^\rho}\|_{\varphi_{r}}.
\]
Moreover,
\[
\Ex\|(c_iX_i)_{i> e^\rho}\|_\rho\leq \bigl(\Ex\|(c_iX_i)_{i> e^\rho}\|_\rho^\rho\bigr)^{1/\rho}
=\|X_1\|_\rho\|(c_i)_{i> e^\rho}\|_\rho
\sim \rho^{1/r}\|(c_i)_{i> e^\rho}\|_\rho.
\]
Therefore the upper bound in \eqref{eq:ellrhoweibull} easily follows.

Now we will show the lower bound. By \eqref{eq:ellinftyweibull} we have
\begin{equation}
\label{eq:ellrholower1}
\Ex\|(c_iX_i)_{i=1}^k\|_\rho\geq \Ex\|(c_iX_i)_{i \leq e^\rho}\|_\infty
\sim \|(c_i)_{i\leq e^\rho}\|_{\varphi_{r}},
\end{equation}
so it is enough to show that
\begin{equation}
\label{eq:ellrhotwoshow}
\Ex\|(c_iX_i)_{i=1}^k\|_\rho\gtrsim \rho^{1/r} \|(c_i)_{i> e^\rho}\|_\rho.
\end{equation}

Observe that
\[
\Ex\|(c_iX_i)_{i=1}^k\|_\rho=\Ex\|(c_i|X_i|)_{i=1}^k\|_\rho\geq \|(c_i\Ex|X_i|)_{i=1}^k\|_\rho
=\Ex|X_1|\|c\|_{\rho} \gtrsim \|c\|_{\rho}.
\]
In particular, \eqref{eq:ellrhotwoshow} holds for $\rho\leq 2$.

Let $C_1$ be a suitably chosen constant (to be fixed later). If
$\rho^{1/r}\|(c_i)_{i> e^\rho}\|_\rho\leq C_1\|(c_i)_{i\leq e^\rho}\|_{\varphi_{r}}$, then
\eqref{eq:ellrholower1} yields \eqref{eq:ellrhotwoshow}.
Thus we may assume that $\rho>2$ and $\rho^{1/r}\|(c_i)_{i> e^\rho}\|_\rho> C_1\|(c_i)_{i\leq e^\rho}\|_{\varphi_{r}}$.

The variables $X_i$ have log-concave tails, hence \cite[Theorem~1]{LaLCT} yields
\[
\Ex \|(c_iX_i)_{i>e^\rho}\|_{\rho} 
\ge \frac{1}{C}\bigl( \Ex \|(c_iX_i)_{i>e^\rho}\|_\rho^\rho\bigr)^{1/\rho} 
- \sup_{t\in B_{\rho^*}^k} \biggl\| \sum_{i>e^\rho} t_ic_i X_i \biggr\|_\rho.
\]
We have
\[
\bigl( \Ex \|(c_iX_i)_{i>e^\rho}\|_\rho^\rho\bigr)^{1/\rho} 
\sim \rho^{1/r}\|(c_i)_{i>e^\rho}\|_\rho.
\]
Inequality \eqref{eq:comp-moms-sums-Weibull}
 and Lemma \ref{lem:multiplier} (applied with $\rho_1=\rho^*\leq 2$ and $\rho_2\in \{2,r^*\}$)
yield
\begin{align*}
\sup_{t\in B_{\rho^*}^n} \biggl\| \sum_{i>e^\rho} t_ic_i X_i \biggr\|_\rho
&\lesssim (\rho^{1/r}+\rho^{1/2})\max_{i>e^\rho}|c_i|\leq 2 \rho^{1/r}c_{\lceil e^\rho\rceil}
\leq  2 \rho^{1/r}\|(c_i)_{i\leq e^\rho}\|_{\varphi_{r}}
\varphi_r^{-1}\Bigl(\frac{1}{\lfloor e^\rho\rfloor}\Bigr)
\\
&\lesssim \|(c_i)_{i\leq e^\rho}\|_{\varphi_{r}}.
\end{align*}
Therefore 
\[
\Ex \|(c_iX_i)\|_{\rho} 
\geq \frac{1}{C_2}\rho^{1/r}\|(c_i)_{i>e^\rho}\|_\rho-C_3\|(c_i)_{i\leq e^\rho}\|_{\varphi_{r}}
\geq \Big(\frac{1}{C_2}-\frac{C_3}{C_1}\Big)\rho^{1/r}\|(c_i)_{i>e^\rho}\|_\rho.
\]
So to get \eqref{eq:ellrhotwoshow} and conclude the proof it is enough to choose $C_1=2C_2C_3$.
\end{proof}

We shall also use the following lemma which is standard, but  we prove it for the sake of completeness. 
\begin{lem}
\label{lem:WeibGauss}
Let $(X_i)_{i=1}^k$ be iid Weibull r.v.'s with parameter $2$. Then for any norm $\|\cdot\|$ 
on $\er^k$ we have
\[
\Ex\|(X_i)_{i=1}^k\|\sim \Ex\|(g_i)_{i=1}^k\|.
\] 
Moreover, if $(Y_i)_{i=1}^k$ are iid Weibull r.v.'s with parameter $r\in[1, 2]$, then
for any norm  $\|\cdot\|$  on $\er^k$ we have
\begin{equation}	
\label{eq:WeibGaussIneq}
\Ex\|(Y_i)_{i=1}^k\|\gtrsim \Ex\|(g_i)_{i=1}^k\|.
\end{equation}
\end{lem}

\begin{proof} We have $\Pr(|g_i|\geq t)\leq e^{-t^2/2}=\Pr(|\sqrt{2}X_i|\geq t)$. Thus we may
find such a representation of $X_i$'s and $g_i$'s  that $|g_i|\leq |\sqrt{2}X_i|$ a.s.
Let $(\ve_i)_{i=1}^k$ be a sequence of iid symmetric $\pm 1$ r.v.'s (Rademachers) independent of
$(X_i)_{i=1}^k$ and $(g_i)_{i=1}^k$. Then  the contraction principle implies
\[
 \Ex\|(g_i)_{i=1}^k\|= \Ex\|(\ve_i|g_i|)_{i=1}^k\|\leq
\Ex\|(\ve_i|\sqrt{2}X_i|)_{i=1}^k\|=\sqrt{2}\Ex\|(X_i)_{i=1}^k\|.
\]

To justify the opposite inequality  observe that there exists $c>0$ (one may take $c=1/4)$ such that
for all $t\geq 0$, $\Pr(|g_i|\geq t)\geq ce^{-t^2}$ and  proceed  similarly as in the proof of  Lemma~\ref{lem:Weibull-to-psi_r}.

Using the inequality $e^{-t^2/2}\leq Ce^{-t^r/2}$ for $t\geq 0$ (one may take $C=\sqrt e$) and proceeding  in a similar way as above we may  prove \eqref{eq:WeibGaussIneq}.
\end{proof}

Proposition~\ref{prop:ellrhoweibull}, Remark~\ref{rem:ellsmallrhoweibull} and Lemma~\ref{lem:WeibGauss} yield the following bound
for $\ell_\rho$-norms of a~Gaussian sequence.

\begin{cor}
\label{cor:ellrhogauss}
For every $1\leq \rho\leq  \infty$ and every sequence $c=(c_i)_{i\leq k}$ we have
\begin{equation*}
\Ex\bigl\|(c_ig_i)_{i=1}^k\bigr\|_\rho\sim
\|(c_i^*)_{i\leq e^\rho}\|_{\varphi_{2}}
+\sqrt{\rho}\|(c_i^*)_{i> e^\rho}\|_\rho.
\end{equation*}
In particular, for $\rho\leq 2$ we have
\[
\Ex\bigl\|(c_ig_i)_{i=1}^k\bigr\|_\rho\sim \|c\|_\rho.
\]
\end{cor}

\begin{rem}
\label{rem:estellpGauss}
 In the case $c_i=1$ Corollary~\ref{cor:ellrhogauss} yields the well known bound
\[
\Ex\bigl\|(g_i)_{i=1}^k\bigr\|_\rho\sim
\begin{cases}
\rho^{1/2}k^{1/\rho},& 1\leq \rho\leq \Log k,
\\
(\Log k )^{1/2},& \rho\geq \Log k
\end{cases}
\sim (\rho\wedge \Log k)^{1/2} k^{1/\rho}. 
\]
\end{rem}

\begin{proof}[Proof of Theorem \ref{thm:aibjgauss}]
Chevet's inequality \eqref{eq:Chevet}, appplied with $ S=\{(a_is_i)\colon s\in B_{q^*}^m\}$ and 
$T=\{(b_jt_j)\colon t\in B_{p}^n\}$, yields
\[
\Ex\bigl\|(a_ib_jg_{i,j})_{i\leq m,j\leq n}\bigr\|_{\ell_p^n\rightarrow\ell_q^m}
\sim \sup_{s\in B_{q^*}^m}\|(a_is_i)\|_2\Ex\|(b_jg_j)\|_{p^*}
+\sup_{t\in B_{p}^n}\|(b_jt_j)\|_2\Ex\|(a_ig_i)\|_{q}.
\]
Lemma \ref{lem:multiplier} and Corollary \ref{cor:ellrhogauss} yield the assertion.
\end{proof}

\begin{proof}[Proof of Theorem \ref{thm:aibjWeibull}]
Theorem \ref{thm:chevetWeibull}, appplied with $S=\{(a_is_i)\colon s\in B_{q^*}^m\}$ and 
$T=\{(b_jt_j)\colon t\in B_{p}^n\}$, yields
\begin{align*}
\Ex\bigl\|(a_ib_jX_{i,j})_{i\leq m,j\leq n}\bigr\|_{\ell_p^n\rightarrow\ell_q^m}
\sim\, & \Ex\bigl\|(a_ib_jg_{i,j})_{i\leq m,j\leq n}\bigr\|_{\ell_p^n\rightarrow\ell_q^m}
\\
&\qquad
+\sup_{s\in B_{q^*}^m}\|(a_is_i)\|_{r^*}\Ex\|(b_jX_j)\|_{p^*}
+\sup_{t\in B_{p}^n}\|(b_jt_j)\|_{r^*}\Ex\|(a_iX_i)\|_{q}.
\end{align*}
To get the assertion we use Theorem \ref{thm:aibjgauss}, Lemma \ref{lem:multiplier},  Proposition \ref{prop:ellrhoweibull}, and Remark~\ref{rem:ellsmallrhoweibull} together with the following observations:
\begin{itemize}
\item for $q<r\leq2$  we have $\|a\|_{2q^*/(q^*-2)}\geq \|a\|_{r^*q^*/(q^*-r^*)}$, 
and for $p^*<r\leq 2$ we have $\|b\|_{2p/(p-2)}\geq \|b\|_{r^*p/(p-r^*)}$,
\item $\Ex\|(a_iX_i)_{i=1}^m\|_{q}\gtrsim \Ex\|(a_ig_i)_{i=1}^m\|_{q}$ and 
$\Ex\|(b_jX_j)_{j=1}^n\|_{p^*} \gtrsim \Ex\|(b_jg_j)_{j=1}^n\|_{p^*}$, which follows by inequality \eqref{eq:WeibGaussIneq}.	\qedhere
\end{itemize}
\end{proof}

\section{Operator norms of submatrices} 
\label{sect:submatrices}

In this section we  prove Theorem~\ref{thm:submatricespsir} about the norms of submatrices.
To prove it we shall use Theorem \ref{thm:chevetWeibull} and
Corollary \ref{cor:psi-r-Chevet}. Thus, we need to estimate
\[
\Ex\sup_{|I|=k}\bigl(\sum_{i\in I}|X_i|^q\bigr)^{1/q}=\Ex \Bigl(\sum_{i=1}^k (X_i^*)^q\Bigr)^{1/q},
\]
where $(X_1^*,X_2^*,\ldots,X_m^*)$ denotes the non-increasing rearrangement of
$(|X_1|,\ldots,|X_m|)$. This is done in the next two technical lemmas.

\begin{lem}
\label{lem:ellqq}
Let $X_1,\ldots,X_m$ be iid symmetric Weibull r.v.'s with shape parameter $r\in [1,2]$.
Then for every $q\geq 1$ and $1\leq k\leq m$ we have
\[
\Bigl(\Ex \sum_{i=1}^k (X_i^*)^q\Bigr)^{1/q} \sim k^{1/q} \Bigl(\Log\Bigl(\frac{m}{k}\Bigr)\vee q\Bigr)^{1/r}.
\]
\end{lem}

\begin{proof}
By, say, \cite[Theorem~3.2]{LS-order}, we get
\[
\Ex  \sum_{i=1}^k (X_i^*)^q \sim kt_*,\quad \mbox{where}\quad 
t_*:=\inf\Bigl\{t>0\colon \Ex |X_1|^q I_{\{|X_1|^q>t\}} \le t\frac{k}{m}  \Bigr\}.
\]

Let $t_1:=(2q+\ln(\frac{m}{k}))^{q/r}$. Then
\begin{align*}
\Ex |X_1|^q I_{\{|X_1|^q>t_1\}}
&\leq \sum_{l=0}^\infty e^{(l+1)q}t_1\Pr(|X_1|>e^lt_1^{1/q})
= t_1\sum_{l=0}^\infty e^{(l+1)q}e^{-e^{lr}t_1^{r/q}}
\\
&\leq t_1e^q\sum_{l=0}^\infty e^{lq}e^{-(1+lr)(2q+\ln(\frac{m}{k}))}
\leq t_1e^{-q}\frac{k}{m}\sum_{l=0}^\infty e^{-(2r-1)ql}< t_1\frac{k}{m}.
\end{align*}
Thus, $t_*^{1/q}\leq t_1^{1/q}\sim (\Log(\frac{m}{k})\vee q)^{1/r}$.

Let $t_2:=\log^{q/r}(m/k)$ and $t_3=\frac{1}{2}\Ex |X_1|^q =  \frac 12 \Gamma( \frac qr +1)$. 
We have
\[
\Ex |X_1|^q I_{\{|X_1|^q>t_2\}}> t_2\Pr(|X_1|>t_2^{1/q})
= t_2e^{-t_2^{r/q}}
=t_2\frac{k}{m},
\]
and
\[
\Ex |X_1|^q I_{\{|X_1|^q>t_3\}}=\Ex |X_1|^q-\Ex |X_1|^q I_{\{|X_1|^q\leq t_3\}}
> \frac{1}{2}\Ex |X_1|^q\geq t_3\frac{k}{m}.
\]
Therefore, $t_*^{1/q}\geq (t_2\vee t_3)^{1/q}\sim (\Log(\frac{m}{k})\vee q)^{1/r}$.
\end{proof}

\begin{lem}
\label{lem:exellq}
Let $X_1,\ldots,X_m$ be iid symmetric Weibull r.v.'s with shape parameter $r\in [1,2]$.
Then for $q\geq 1$, $1\leq k\leq m$ we have
\[
\Ex \Bigl(\sum_{i=1}^k (X_i^*)^q\Bigr)^{1/q} \sim 
\begin{cases}
\Log^{1/r} m& q\geq \Log k
\\
k^{1/q} \bigl(\Log\bigl(\frac{m}{k}\bigr)\vee q\bigr)^{1/r} & q<\Log k
\end{cases}
\sim k^{1/q}\Bigl(\Log\Bigl(\frac{m}{k}\Bigr)\vee (q\wedge \Log k)\Bigr)^{1/r}.
\]
\end{lem}

\begin{proof}
If $q\ge \Log k$, then
\[
\Ex \Bigl( \sum_{i=1}^k (X_i^*)^q \Bigr)^{1/q}  \sim \Ex \max_{i\le m} |X_i|\sim \Log^{1/r} m,
\]
where the last (standard) bound follows e.g.\ by Lemma \ref{lem:ellqq} applied with $k=q=1$.

If $q\in [1,2]$ then Khinchine-Kahane-type inequality (cf., \cite[Corollary 1.4]{LatalaStrzeleckaMat}) and Lemma \ref{lem:ellqq} yield
\[
\Ex \Bigl(\sum_{i=1}^k (X_i^*)^q\Bigr)^{1/q} \sim
\Bigl(\Ex \sum_{i=1}^k (X_i^*)^q\Bigr)^{1/q}\sim k^{1/q} \Log^{1/r}\Bigl(\frac{m}{k}\Bigr).
\]

From now on assume that $2\leq q \leq \Log k$. By Lemma \ref{lem:ellqq}
\[
\Ex \Bigl( \sum_{i=1}^k (X_i^*)^q \Bigr)^{1/q}\leq \Bigl( \Ex \sum_{i=1}^k (X_i^*)^q \Bigr)^{1/q}
\sim k^{1/q}\Bigl(\Log\Bigl(\frac{m}{k}\Bigr)\vee q\Bigr)^{1/r}.
\]
Variables $X_i$ have log-concave tails, so by \cite[Theorem 1]{LaLCT},
\begin{align*}
\Ex \Bigl( \sum_{i=1}^k (X_i^*)^q \Bigr)^{1/q}
&=\Ex\sup_{|I|=k} \sup_{s\in B_{q^*}^I} \sum_{i\in I} s_iX_i 
\\
&\geq \frac{1}{C_1}\Bigl\|\sup_{|I|=k} \sup_{s\in B_{q^*}^I} \sum_{i\in I} s_iX_i \Bigr\|_q
-\sup_{|I|=k} \sup_{s\in B_{q^*}^I} \Bigl\|\sum_{i\in I} s_iX_i\Bigr\|_q.
\end{align*}
Since $q^*\leq 2\leq r^*$,  \eqref{eq:comp-moms-sums-Weibull} implies that
 for any $I\subset [m]$ and any $s\in B_{q^*}^I$ we have ,
\[
\Bigl\|\sum_{i\in I} s_iX_i\Bigr\|_q\lesssim  q^{1/r}\|s\|_{r^*}+q^{1/2}\|s\|_2
\leq q^{1/r}+q^{1/2}\le 2 q^{1/r}.
\]
This together with Lemma \ref{lem:ellqq}  yields
\[
\Ex \Bigl( \sum_{i=1}^k (X_i^*)^q \Bigr)^{1/q}
\geq \frac{1}{C_1C_2}k^{1/q}\Bigl(\Log\Bigl(\frac{m}{k}\Bigr)\vee q\Bigl)^{1/r}-C_3q^{1/r}.
\]
Thus if $k\geq (2C_1C_2C_3)^q$ we get 
$\Ex( \sum_{i=1}^k (X_i^*)^q )^{1/q}\geq \frac{1}{2C_1C_2}k^{1/q}(\Log(\frac{m}{k})\vee q)^{1/r}$.
If $k\leq (2C_1C_2C_3)^q$ then $k^{1/q}\sim 1$ and
$\Ex( \sum_{i=1}^k (X_i^*)^q )^{1/q}\geq (\Ex (X_1^*)^q)^{1/q}\sim (\Log m\vee q)^{1/r}$.
\end{proof}

\begin{proof}[Proof of Theorem \ref{thm:submatricespsir}]
We  use Theorem~\ref{thm:chevetWeibull} and
Corollary~\ref{cor:psi-r-Chevet} with
\[
S=\bigcup_{I}B_{q*}^I,\quad T=\bigcup_{J}B_{p}^{J},
\]
where $B_{q*}^I$ is the unit ball in the space $\ell_{q^*}^I$, and the sums run over, 
respectively, all sets $I\subset[m]$ and $J\subset[n]$ such that $|I|=k$ and $|J|=l$.
We only need to estimate the quantities on the right-hand side of the two-sided bounds from 
Theorem~\ref{thm:chevetWeibull} and Corollary~\ref{cor:psi-r-Chevet}.
We have for $\rho\in\{2,r^*\}$,
\[
\sup_{s\in S}\|s\|_{\rho}=k^{(1/q-1/\rho^*)\vee 0},
\quad \sup_{t\in T}\|t\|_{\rho}=l^{(1/p^*-1/\rho^*)\vee 0}.
\]		
Lemmas \ref{lem:exellq} and \ref{lem:WeibGauss} yield
\begin{align*}
\Ex\sup_{s\in S}\sum_{i=1}^ms_iX_i  = \Ex \Bigl(\sum_{i=1}^k (X_i^*)^q\Bigr)^{1/q}
&\sim k^{1/q} \Bigl(\Log\Bigl(\frac{m}{k}\Bigr)\vee (q\wedge \Log k)\Bigr)^{1/r},
\\
\Ex\sup_{s\in S}\sum_{i=1}^ms_ig_i = \Ex \Bigl(\sum_{i=1}^k (g_i^*)^q\Bigr)^{1/q}
&\sim k^{1/q} \Bigl(\Log\Bigl(\frac{m}{k}\Bigr)\vee (q\wedge \Log k)\Bigr)^{1/2}.
\end{align*}
Similarily,
\begin{align*}
\Ex\sup_{t\in T}\sum_{i=1}^nt_iX_i = \Ex \Bigl(\sum_{i=1}^l (X_i^*)^{p^*}\Bigr)^{1/p^*}
&\sim l^{1/p^*} \Bigl(\Log\Bigl(\frac{n}{l}\Bigr)\vee (p^*\wedge \Log l)\Bigr)^{1/r},
\\
\Ex\sup_{t\in T}\sum_{i=1}^n t_ig_i = \Ex \Bigl(\sum_{i=1}^l (g_i^*)^{p^*}\Bigr)^{1/p^*}
&\sim l^{1/p^*} \Bigl(\Log\Bigl(\frac{n}{l}\Bigr)\vee (p^*\wedge \Log l)\Bigr)^{1/2}.
\qedhere
\end{align*}	
\end{proof}

\vspace{0.4cm}

  \bibliographystyle{amsplain}
  \bibliography{matrices}

\end{document}